\documentclass{amsart}
\usepackage [utf8] {inputenc}
\usepackage{mathtools}
\usepackage{amsthm}
\usepackage{amssymb}
\usepackage{url}

\newtheorem{theorem}{Theorem}[section]
\newtheorem{proposition}[theorem]{Proposition}
\newtheorem{lemma}[theorem]{Lemma}
\newtheorem{corollary}[theorem]{Corollary}

\DeclareMathOperator{\Aut}{Aut}
\DeclareMathOperator{\Sym}{Sym}

\DeclareMathOperator{\supp}{supp}
\DeclareMathOperator{\motion}{motion}
\DeclareMathOperator{\mindeg}{mindeg}
\DeclareMathOperator{\X}{\mathfrak{X}}

\title{On the automorphism group of a distance-regular graph}

\author{L\'aszl\'o Pyber}
\address{HUN-REN Alfr\'ed R\'enyi Institute of Mathematics, Realtanoda utca 13-15, H-1053, Budapest, Hungary}
\email{pyber.laszlo@renyi.hu}

\author{Saveliy V. Skresanov}
\address{Sobolev Institute of Mathematics, 4 Acad. Koptyug avenue, Novosibirsk, Russia}
\email{skresan@math.nsc.ru}

\keywords{distance-regular graph, motion, diameter, automorphism group}
\thanks{This work on the project leading to this application has received funding from the European Research Council (ERC)
under the European Union’s Horizon 2020 research and innovation programme (grant agreement No. 741420).
The first named author was partly supported by the National Research, Development and Innovation Office (NKFIH) Grant No. K138596.
The research of the second named author was carried out within the framework of the Sobolev Institute of Mathematics state contract (project FWNF-2022-0002).}

\begin{document}

\begin{abstract}
	The motion of a graph is the minimal degree of its full automorphism group.
	Babai conjectured that the motion of a primitive distance-regular graph on \( n \) vertices of diameter greater than two
	is at least \( n/C \) for some universal constant \( C > 0 \), unless the graph is a Johnson or Hamming graph.
	We prove that the motion of a distance-regular graph of diameter \( d \geq 3 \) on \( n \) vertices
	is at least \( Cn/(\log n)^6 \) for some universal constant \( C > 0 \), unless it is a Johnson, a Hamming or a crown graph.
	This follows using an improvement of an earlier result by Kivva who gave a lower bound on motion of the form \( n/c_d \), where \( c_d \) depends exponentially on~\( d \).
	As a corollary we derive a quasipolynomial upper bound for the automorphism group of a primitive distance-regular graph
	acting edge-transitively on the graph and on its distance-2 graph.
	The proofs use elementary combinatorial arguments and do not depend on the classification of finite simple groups.
\end{abstract}

\maketitle

\section{Introduction}

The degree of a permutation \( g \) of the set \( V \) is the number of points of \( V \) not fixed by~\( g \).
Given a permutation group \( H \) on \( V \), let the minimal degree \( \mindeg(H) \) be the minimum of degrees
of all nonidentity elements of \( H \). Describing finite primitive permutation groups with small minimal degree is a classical
problem in finite group theory which goes back to the works of Jordan and Bochert. In~\cite{bochert} Bochert proved
that a doubly transitive permutation group of degree \( n \) has minimal degree at least \( n/4 - 1\) with known exceptions.
Using the classification of finite simple groups (CFSG), Liebeck and Saxl~\cite{liebeckSaxl} classified primitive permutation groups
of degree \( n \) with minimal degree less than \( n/3 \), and Guralnick and Magaard~\cite{guralnik} further strengthened this result
to a description of primitive permutation groups with minimal degree at most \( n/2 \).

Given a combinatorial object \( G \) (a graph or a coherent configuration) let the motion of \( G \) be \( \motion(G) = \mindeg(\Aut(G)) \).
As an attempt to generalize the above results to a combinatorial setting, Babai proved the following result on motion
of distance-regular graphs of diameter two, that is, strongly-regular graphs:
\begin{proposition}[{\cite{babaiAut, babaiAut2}}]\label{babaiSRG}
	Let \( G \) be a strongly-regular graph on \( n \geq 29 \) vertices. Then either
	\[ \motion(G) \geq \frac{n}{8}, \]
	or \( G \) or its complement is a Johnson graph, a Hamming graph or a disjoint union of cliques of the same size.
\end{proposition}

Furthermore, Babai proposed a conjecture, which gives a natural generalization of the above result to distance-regular graphs of arbitrary diameter:
\medskip

\noindent\textbf{Conjecture.}
(Babai,~\cite[Conjecture~6.6]{kivvaJH}) \emph{There exists \( \gamma > 0 \) such that for any primitive distance-regular graph \( G \)
of diameter greater than two on \( n \) vertices either
\[ \motion(G) \geq \gamma n, \]
or \( G \) is a Johnson or a Hamming graph.}
\medskip

The first step to settling the conjecture was made by Kivva, who proved it for distance-regular graphs of bounded diameter, i.e.\
in his result the constant~\( \gamma \) is allowed to depend on the diameter of the graph.
In fact, he was able to establish his result both for primitive and imprimitive distance-regular graphs.
Recall that a \emph{crown graph} is a graph obtained from a regular complete bipartite graph by removing a perfect matching.

\begin{proposition}[{\cite[Theorem~1.12]{kivvaJH}}]\label{kivvaOld}
	For any \( d \geq 3 \) there exists \( \gamma_d > 0 \) such that for any distance-regular graph \( G \)
	of diameter \( d \) on \( n \) vertices either
	\[ \motion(G) \geq \gamma_d n \]
	or \( G \) is a Johnson graph, a Hamming graph or a crown graph.
\end{proposition}

Bounds on \( \motion(G) \) have strong structural consequences for the structure of the automorphism group of \( G \).
The thickness \( \theta(H) \) of a group \( H \) is the largest degree of an alternating section of \( H \).
It follows from a result of Wielandt~\cite{wielandtD} (see Proposition~\ref{thetaIneq}), that if \( \mindeg(H) \geq \alpha n \),
where \( n \) is the degree of \( H \) and \( \alpha > 0 \), then \( \theta(H) \leq \frac{3}{\alpha} \log n \).
In particular, Kivva's result has the following immediate consequence.

\begin{corollary}[{\cite[Theorem~6.3]{kivvaJH}}]
	For any \( d \geq 3 \) there exists \( \alpha_d > 0 \) such that for any distance-regular graph \( G \)
	of diameter \( d \) on \( n \) vertices either
	\[ \theta(G) \leq \alpha_d \log n \]
	or \( G \) is a Johnson graph, a Hamming graph or a crown graph.
\end{corollary}

One shortcoming of Kivva's result is that, in the notation of Proposition~\ref{kivvaOld}, the constant \( 1/\gamma_d \) grows exponentially in~\( d \),
see~\cite[Remark~4.9]{kivvaGap}. In particular, even though it follows from a result of Pyber~\cite{pyberdrg} (see Proposition~\ref{diam}) that
the diameter \( d \) of a distance-regular graph on \( n \) vertices of valency at least~3 is bounded above by \( 5 \log n \), we cannot obtain a polylogarithmic
bound on the thickness of the automorphism group from Proposition~\ref{kivvaOld}. Our main result is an improvement of Kivva's theorem, namely,
we show that \( 1/\gamma_d \) can be taken to be a polynomial in \( d \).

\begin{theorem}\label{main}
	Let \( G \) be a distance-regular graph on \( n \) points of diameter~\( d \geq 3 \).
	Then either \( G \) is a Johnson, Hamming or crown graph, or
	\[ \motion(G) \geq C \frac{n}{d^6} \]
	for some universal constant \( C > 0 \).
\end{theorem}

As mentioned earlier, \( d \leq 5 \log n \) if \( G \) has valency at least~3, hence in this case the motion
is bounded below by \( C n /(5 \log n)^6 \).
Now using Wielandt's bound for thickness, we immediately obtain the following polylogarithmic estimate.

\begin{corollary}
	Let \( G \) be a distance-regular graph on \( n \) points of diameter~\( d \geq 3 \).
	Then either \( G \) is a Johnson, Hamming or crown graph, or
	\[ \theta(G) \leq C (\log n)^7 \]
	for some universal constant \( C > 0 \).
\end{corollary}

It follows from CFSG and a result of Babai, Cameron and P\'alfy~\cite{bcp} that the order of a primitive permutation
group \( H \) is bounded above in terms of its degree and thickness by \( n^{\theta(H)^c} \) for some universal constant \( c > 0 \).
In particular, this implies a quasipolynomial bound for a primitive group of automorphisms of a distance-regular graph, which is not a Johnson, Hamming or a crown graph.
Our second main result is an improvement to this observation, which does not require CFSG.
For a connected graph \( G \) on \( V \) let \( G_2 \) denote the distance~2 graph, i.e.\
\[ G_2 = \{ \{x, y \} \subseteq V \mid d(x, y) = 2 \}. \]
It turns out that instead of assuming that the group of automorphisms is primitive, it suffices to assume that the group acts edge-transitively on \( G \) and \( G_2 \).

\begin{theorem}\label{maingroup}
	Let \( G \) be a primitive distance-regular graph on \( n \) points of diameter greater than two,
	which is not a Johnson or a Hamming graph.
	Assume that its full automorphism group \( H \) acts edge-transitively on \( G \) and \( G_2 \).
	Then \( H \) has a base of size at most \( C(\log n)^9 \) for some universal constant \( C > 0 \),
	in particular, \( |H| \leq n^{C(\log n)^9} \).
\end{theorem}

It is desirable to prove a similar bound for arbitrary groups of automorphisms of distance-regular graphs,
i.e.\ without necessarily assuming edge-transitivity on \( G \) or \( G_2 \). Note that Babai~\cite{babaiAut2} conjectures
that such a result holds for strongly-regular graphs. This conjecture ultimately aims at obtaining an elementary, that is, CFSG-free, proof
of Cameron's theorem~\cite[Theorem~6.1]{cameron} on the orders of primitive permutation groups, and our Theorem~\ref{maingroup} is a step in this direction.

The main novelty of our proof lies in the following improvement of~\cite[Theorem~1.7]{kivvaGap}.

\begin{theorem}\label{geom}
	Let \( G \) be a primitive distance-regular graph on \( n \) vertices with diameter \( d \geq 3 \).
	Then either \( \motion(G) \geq n/(40d^5) \) or \( G \) is Delsarte-geometric with the smallest eigenvalue at least \( -5d \).
\end{theorem}

The original result of Kivva~\cite{kivvaGap} relies on a series of elaborate inequalities, involving Ostrowski's theorem on matrix eigenvalues
and the ``growth-induced tradeoff''~\cite[Theorem~1.2]{kivvaGap}. Our proof offers both a significant simplification of the argument
and a better bound on the motion. The main tools for us are the Bang-Koolen theorem~\cite{bang}, Metsch theorem~\cite{metsch}
and a generalization of Babai-Szegedy's~\cite{szegedy} edge expansion bound to coherent configurations.
Namely, in Proposition~\ref{expand} we prove that edge expansion of a connected symmetric relation of a homogeneous coherent configuration
is bounded below by \( k/(2d) \), where \( k \) is the degree and \( d \) is the diameter of the relation. As an application of that bound,
in Proposition~\ref{eigenexp} we show that the second largest eigenvalue \( \theta \) of a distance-regular graph of degree \( k \)
and diameter \( d \) satisfies \( \theta \leq k(1 - 1/(8d^2)) \), which may be of independent interest.
We believe that extending other results on vertex- and edge-transitive graphs to coherent configurations may be an interesting area of research.

In the proof of Theorem~\ref{main} we use Theorem~\ref{geom} together with deep results of Kivva on characterizations of Johnson and Hamming graphs~\cite{kivvaJH}.

The proof of Theorem~\ref{maingroup} relies on a small combinatorial trick (Lemma~\ref{pybTrick}) and, surprisingly, on the argument
of the first author showing that the orders of doubly transitive permutation groups not containing the alternating group have quasipolynomial size~\cite{pyber2tr}.

We mention that, similarly to distance-regular graphs, Babai conjectured that the motion of a primitive coherent configuration
of degree \( n \) should be at least \( \gamma n \) for some universal constant \( \gamma > 0 \), unless the configuration is a so-called Cameron scheme,
see~\cite[Conjecture~6.5]{kivvaJH} and~\cite[Conjecture~12.1]{babaiAlgo}.
If true, this would give a combinatorial analogue of Liebeck and Saxl's description of primitive permutation groups of small minimal degree. The case of rank~3 coherent configurations
was settled by Babai in~\cite{babaiAut, babaiAut2} and the case of rank~4 by Kivva in~\cite{kivvaRank4}. Yet, recently this conjecture was refuted
by Eberhard~\cite{eberhard} who proposed a new construction of primitive coherent configurations with small motion.
The motion conjecture for distance-regular graphs therefore stands as an important open case.

The structure of the paper is as follows. In Section~\ref{prel} we provide preliminary results and notation for distance-regular
graphs (Subsection~\ref{pDrg}), coherent configurations (Subsection~\ref{pCC}), and recall Kivva's characterizations of Johnson and Hamming graphs (Subsection~\ref{pJH}).
In Section~\ref{pgraphs} we prove Theorems~\ref{main} and~\ref{geom}, and in Section~\ref{pgroups} we establish Theorem~\ref{maingroup}.

\section{Preliminaries}\label{prel}

\subsection{Distance-regular graphs}\label{pDrg}

Let \( G \) be a connected graph on \( n \) points, and let \( V \) denote the set of vertices.
For two vertices \( x, y \in V \) let \( d(x, y) \) denote the distance between \( x \) and \( y \),
and let \( \Gamma_i(x) \) be the set of \( y \) with \( d(x, y) = i \).

For vertices \( x \) and \( y \) at distance \( j \) write \( c_j(x, y) = | \Gamma_{j-1}(x) \cap \Gamma_1(y) | \),
\( a_j(x, y) = | \Gamma_j(x) \cap \Gamma_1(y) | \) and \( b_j(x, y) = | \Gamma_{j+1}(x) \cap \Gamma_1(y) | \).
The graph \( G \) is called \emph{distance-regular}, if these numbers depend only on the distance \( j \) between \( x \) and \( y \).
In this case we will simply write \( c_j \), \( a_j \) and \( b_j \). If \( k \) is the valency of \( G \), then for all \( j \) we have \( a_j = k - b_j - c_j \).

If \( d \) is the diameter of \( G \), then we have the following inequalities
\begin{align*}
	k &= b_0 > b_1 \geq b_2 \geq \dots \geq b_{d-1} \geq 1,\\
	1 &= c_1 \leq c_2 \leq \dots \leq c_d \leq k.
\end{align*}
As distance-regular graphs of valency \( 2 \) are cycles, we will assume that \( k > 2 \).

The numbers \( k_i = |\Gamma_i(x)| \), \( i = 1, \dots, d \), do not depend on the vertex \( x \in V \) and satisfy
\[ \frac{k}{c_2} = \frac{b_0}{c_2} > \frac{b_1}{c_2} = \frac{k_2}{k_1} \geq \dots \geq \frac{k_{d}}{k_{d-1}}. \]
We may also always assume that \( k_1 \leq k_2 \)~\cite[Proposition~5.1.1]{brouwer}.

In the following \( \log n \) is a logarithm in base~2.
\begin{proposition}[{\cite{pyberdrg}}]\label{diam}
	The diameter of a distance-regular graph on \( n \) vertices and valency greater than two is at most \( 5 \log n \).
\end{proposition}
In~\cite[Theorem~1.5]{bangDiam} this bound was improved to \( \frac{8}{3} \log n \), but in our proof the multiplicative constant
will not play a significant role.

The adjacency matrix of a distance-regular graph has precisely \( d+1 \) distinct eigenvalues.
All of them are real, the largest being \( k \). Let \( \theta \) denote the second largest eigenvalue,
and let \( m \) be the smallest eigenvalue. We have \( m < 0 \).

\begin{lemma}[Delsarte bound {\cite{delsarte}, \cite[p.~276]{godsil}}]
	In a distance-regular graph the size of every clique is bounded above by \( 1 - k/m \).
\end{lemma}

A clique of order \( 1 - k/m \) is called a \emph{Delsarte clique}. A clique geometry on \( G \)
is a set of maximal cliques such that every edge of \( G \) is contained in exactly one clique from the set.
A graph is called \emph{Delsarte-geometric} if it has a clique geometry consisting of Delsarte cliques.

\begin{lemma}[{\cite[Lemma~2.15]{kivvaJH}}]\label{mint}
	Let \( G \) be a Delsarte-geometric distance-regular graph of diameter at least~\( 2 \)
	with smallest eigenvalue \( m \).  Let \( \mathcal{C} \) be a clique geometry on \( G \) consisting of Delsarte cliques.
	Then \( m \) is an integer and every vertex in \( G \) belongs to precisely \( -m \) cliques from \( \mathcal{C} \).
\end{lemma}

A graph \( G \) is called \emph{sub-amply regular} with parameters \( (n, k, \lambda, \mu) \)
if it is a regular graph on \( n \) vertices of degree \( k \) and every two adjacent vertices
have exactly \( \lambda \) neighbors in common while every two vertices at distance~2
from each other have at most \( \mu \) neighbors in common. A distance-regular graph
is sub-amply regular for \( \lambda = a_1 \) and \( \mu = c_2 \).

\begin{lemma}[{\cite[Lemma~2.5]{kivvaJH}}]\label{mulambda}
	For a distance-regular graph of diameter at least~\( 2 \) one has \( 2\lambda \leq k + \mu \).
\end{lemma}

The following theorem by Bang and Koolen gives a sufficient condition for a distance-regular graph to be Delsarte-geometric.

\begin{proposition}[\cite{bang}]\label{bang}
	If \( \lambda > m^2 \mu \) for a distance-regular graph \( G \),
	then \( G \) is Delsarte-geometric.
\end{proposition}

A result of Metsch~\cite{metsch} gives constraints on parameters of sub-amply regular graphs which
guarantee the existence of a certain clique geometry. Babai and Wilmes~\cite[Theorem~3]{wilmes} gave an asymptoric
version of Metsch's theorem with a simpler proof, which states that a sub-amply regular graph with \( k\mu = o(\lambda^2) \)
with respect to \( n \to \infty \) has a maximal clique of order approximately \( \lambda \).
It will be more convenient for our applications to have an explicit version, which we derive from Metsch's theorem directly.

\begin{proposition}\label{clique}
	Let \( G \) be a sub-amply regular graph with parameters \( (n, k, \lambda, \mu) \)
	such that \( \lambda^2 \geq 4k\mu \). Then \( G \) contains a clique of size at least \( \lambda/2 \).
\end{proposition}
\begin{proof}
	A corollary from the result of Metsch~\cite[Result~2.1]{metsch}, see also~\cite[Corollary~1]{wilmes}, implies that if
	\begin{equation}\label{meq}
		(\lambda+1)^2 > (3k + \lambda + 1)(\mu - 1)
	\end{equation}
	then \( G \) contains a clique of size at least
	\begin{equation}\label{cleq}
		\lambda + 2 - (\lceil (3/2)k/(\lambda+1) \rceil - 1)(\mu - 1).
	\end{equation}
	Since \( \lambda \leq k-1 \), we have \( 3k + \lambda + 1 \leq 4k \). Hence inequality \( \lambda^2 \geq 4k\mu \)
	implies inequality~(\ref{meq}) from Metsch's theorem, and we have a clique of size at least~(\ref{cleq}).
	We claim that
	\[
		\lambda + 2 - (\lceil (3/2)k/(\lambda+1) \rceil - 1)(\mu - 1) \geq \lambda/2 + 2.
	\]
	Indeed, this inequality is follows from a stronger one
	\[
		\frac{\lambda}{2} \geq \frac{3k}{2(\lambda + 1)}(\mu - 1),
	\]
	which is implied by our assumption \( \lambda^2 \geq 4k\mu \).
\end{proof}

\subsection{Coherent configurations and motion}\label{pCC}

Let \( \X = (V, X) \) be a tuple, where \( V \) is a set and \( X \) is a partition of \( V \times V \).
We say that \( \X \) is a \emph{homogeneous coherent configuration} if the following holds~\cite[Definitions~2.1.1 and 2.1.3]{ponom}:
\begin{itemize}
	\item The diagonal \( \{ (v, v) \mid v \in V \} \) lies in \( X \);
	\item For every relation \( R \in X \) its transpose \( \{ (v, u) \mid (u, v) \in R \} \) lies in \( X \);
	\item For any relations \( R, S, T \in X \) and any \( (u, v) \in R \), the number
		\[ c_{ST}^R = | \{ w \in V \mid (u, w) \in S,\, (w, v) \in T \} | \]
		does not depend on the choice of \( (u, v) \in R \).
\end{itemize}
The size of \( X \) is called the rank. The configuration is called \emph{primitive} if all its nondiagonal relations are connected.

If \( G \) is a distance-regular graph of diameter \( d \), then the relations
\[ X_i = \{ (x, y) \in V \times V \mid d(x, y) = i \}, \, i = 0, \dots, d, \]
form a coherent configuration of rank \( d+1 \) on \( V \). A distance-regular graph is
called \emph{primitive} if all its ``at distance \( i \)`` relations \( X_i \), \( i = 1, \dots, d \),
are connected. The corresponding coherent configuration is also primitive.

In a primitive coherent configuration diameters of nondiagonal relations are less than its rank~\cite[Exercise~8.4]{babaiUpd},
so for a primitive distance-regular graph diameters of \( X_i \), \( i = 1, \dots, d \), are at most \( d \).

The following fact is well-known.

\begin{lemma}[{\cite[Exercise~2.7.25]{ponom}, \cite[Exercise~2.2]{babaiUpd}}]\label{paths}
	Let \( \X = (V, X) \) be a homogeneous coherent configuration, and let \( R, R_1, \dots, R_{m-1} \in X \), \( m \geq 2 \)
	and \( (u, w) \in R \). Then the number of tuples \( (v_1, \dots, v_m) \in V^m \) such that \( (v_1, v_m) = (u, w) \) and \( (v_i, v_{i+1}) \in R_i \),
	\( i = 1, \dots, m-1 \) does not depend on the choice of \( (u, w) \in R \).
\end{lemma}

For a subset \( S \subseteq V \) and a graph \( G \) on \( V \) let \( \delta_G(S) \) denote the set of edges joining \( S \) with \( V \setminus S \) in~\( G \).
Recall that a relation \( G \) of a coherent configuration is symmetric if it coincides with its transpose, i.e.\ for every \( (u, v) \in G \) we also have \( (v, u) \in G \).
We can view such a relation as a regular undirected graph on~\( V \).

We need a generalization of a result for edge-transitive graphs from~\cite[Corollary~2.6]{szegedy}.
The proof rests on essentially the same argument as in~\cite{szegedy}.

\begin{proposition}\label{expand}
	Let \( \X = (V, X) \) be a homogeneous coherent configuration, and let \( G \in X \) be a connected symmetric relation of degree~\( k \)
	and diameter~\( d \). Then for every nonempty \( S \subseteq V \) with \( |S| \leq |V|/2 \) we have
	\[ \frac{|\delta_G(S)|}{|S|} \geq \frac{k}{2d}. \]
\end{proposition}
\begin{proof}
	Let \( d(x, y) \) denote the distance between vertices \( x, y \in V \) in the graph~\( G \).
	A path \( x_0, x_1, \dots, x_m \) in \( G \) will be called a geodesic if \( m = d(x_0, x_m) \).
	Note that \( x_m, \dots, x_1, x_0 \) is a different geodesic.

	By~\cite[Theorem~2.6.7]{ponom} the distance-\( i \) graph \( \{ (u, v) \in V \times V \mid d(u, v) = i \} \)
	is a union of some relations from \( X \), in particular, the distance between points \( x \) and \( y \) depends
	only on the relation where \( (x, y) \) lies.

	Let \( p(x, y) \) denote the number of geodesics between \( x \) and \( y \). If \( e \) is some edge of \( G \),
	we claim that the number
	\[ P = \sum_{x, y \in V} \frac{1}{p(x, y)} |\{ L \mid L \text{ is a geodesic from } x \text{ to } y \text{ passing through } e \}| \]
	does not depend on the choice of \( e \in G \). Indeed, by Lemma~\ref{paths}, the number \( p(x, y) \) depends only on the relation
	where \( (x, y) \) lies. The number of geodesics from \( x \) to \( y \) passing through \( e = (z, w) \) can be expressed as
	\[
		N_{(z, w)}(x, y) = 
		\begin{cases}
			p(x, z) \cdot p(w, y), & \text{ if } d(x, z) + d(w, y) = d(x, y) - 1,\\
			0, & \text{ otherwise.}
		\end{cases}
	\]
	Thus, by the paragraph above, this number depends only on the relations where \( (x, z) \), \( (w, y) \) and \( (x, y) \) lie (recall that \( (z, w) \) lies in \( G \)).
	We now have the following formula for \( P \):
	\[ P = \sum_{R, S, T \in X} \sum_{\substack{x, y \in V,\\(x, z) \in R, (w, y) \in S, (x, y) \in T}} \frac{1}{p(x, y)} N_{(z, w)}(x, y). \]
	By Lemma~\ref{paths}, the number of tuples \( (x, y) \) such that \( (x, z) \in R \), \( (w, y) \in S \), \( (x, y) \in T \) does not depend on the choice
	of \( (z, w) \in G \), hence \( P \) does not depend on the choice of \( e \in G \) as well.

	Set \( n = |V| \). We now claim that
	\[ nk \cdot P = \sum_{x, y \in V} d(x, y). \]
	The number of edges in \( G \) is exactly \( nk \), and since \( P \) does not depend on the choice of \( e \in G \) we have
	\begin{multline*}
		nk \cdot P = \sum_{e \in G} \sum_{x, y \in V} \frac{1}{p(x, y)} |\{ L \mid L \text{ is a geodesic from } x \text{ to } y \text{ passing through } e \}|\\
		= \sum_{x, y \in V} \frac{1}{p(x, y)} \sum_{e \in G} |\{ L \mid L \text{ is a geodesic from } x \text{ to } y \text{ passing through } e \}|.
	\end{multline*}
	Since each geodesic from \( x \) to \( y \) passes through \( d(x, y) \) edges (exactly once through each edge) and there are \( p(x, y) \) such geodesics, we have
	\begin{multline*}
		\sum_{x, y \in V} \frac{1}{p(x, y)} \sum_{e \in G} |\{ L \mid L \text{ is a geodesic from } x \text{ to } y \text{ passing through } e \}|\\
		= \sum_{x, y \in V} \frac{1}{p(x, y)} \cdot d(x, y)p(x, y) = \sum_{x, y \in V} d(x, y),
	\end{multline*}
	which proves the claim.

	Now we are ready to prove the main result. Choose some nonempty \( S \subseteq V \) with \( |S| \leq n/2 \). Note that
	for \( x \in S \), \( y \in V \setminus S \) a geodesic from \( x \) to \( y \) must pass through an edge in \( \delta_G(S) \), therefore:
	\[ |S| \cdot |V \setminus S| =
	\sum_{x \in S} \sum_{y \in V \setminus S} \frac{1}{p(x, y)} |\{ L \mid L \text{ is a geodesic from } x \text{ to } y \text{ through } \delta_G(S)\}|. \]
	We extend summation to the whole of \( V \) and derive
	\begin{multline*}
		|S| \cdot |V \setminus S| \leq
		\sum_{x, y \in V} \frac{1}{p(x, y)} |\{ L \mid L \text{ is a geodesic from } x \text{ to } y \text{ through } \delta_G(S)\}|\\
		\leq \sum_{x, y \in V} \frac{1}{p(x, y)} \sum_{e \in \delta_G(S)} |\{ L \mid L \text{ is a geodesic from } x \text{ to } y \text{ through } e\}| \\
		= \sum_{e \in \delta_G(S)} \sum_{x, y \in V} \frac{1}{p(x, y)} |\{ L \mid L \text{ is a geodesic from } x \text{ to } y \text{ through } e\}| \\
		= \sum_{e \in \delta_G(S)} P = |\delta_G(S)| \cdot P = \frac{|\delta_G(S)|}{nk} \sum_{x, y \in V} d(x, y) \leq \frac{dn}{k} \cdot |\delta_G(S)|.
	\end{multline*}
	Finally,
	\[ |\delta_G(S)| \geq |S| \cdot \frac{k}{d} \cdot \frac{n - |S|}{n} \geq |S| \cdot \frac{k}{2d}. \]
	The proposition is proved.
\end{proof}

Note that the proposition above applies to distance-regular graphs.

\begin{proposition}\label{eigenexp}
	Let \( G \) be a distance-regular graph of degree \( k \) and diameter~\( d \).
	Then the second largest eigenvalue satisfies \( \theta \leq k(1 - \frac{1}{8d^2}) \).
\end{proposition}
\begin{proof}
	By Proposition~\ref{expand} the edge expansion \( h \) of \( G \) is bounded below by \( k/(2d) \).
	The result follows from the Cheeger inequality \( h \leq \sqrt{2k(k - \theta)} \).
\end{proof}

Given a coherent configuration \( \X = (V, X) \) and \( x, y \in V \) let \( r(x, y) \) denote the unique
relation in \( X \) containing the pair \( (x, y) \).
For \( x, y \in V \) set \( D(x, y) = \{ z \in V \mid r(x, z) \neq r(y, z) \} \).
Define \( D_{\min}(\X) = \min_{x \neq y} |D(x, y)| \).

The following result is implicit in~\cite{babai, babaiUpd}.

\begin{proposition}\label{dmin}
	Let \( \X \) be a primitive coherent configuration on \( n \) points with maximal valency \( k_{\max} \).
	If all nondiagonal relations of \( \X \) have diameter at most~\( d \), then \( D_{\min}(\X) \geq (n - k_{\max})/d \).
\end{proposition}
\begin{proof}
	Follows from~\cite[Lemma~8.9]{babaiUpd} and~\cite[Corollary~8.12]{babaiUpd}.
\end{proof}

Next result gives a lower bound on \( D_{\min} \) for a configuration corresponding to a distance-regular graph in terms of its parameters.

\begin{proposition}[{\cite[Proposition~4.7]{kivvaGap}}]\label{bandc}
	Let \( G \) be a primitive distance-regular graph on \( n \) vertices of diameter \( d \geq 2 \) and valency \( k \),
	and let \( \X \) be the associated coherent configuration. Let \( \epsilon > 0 \).
	Suppose that for some \( j \in \{ 1, \dots, d-1 \} \) the inequalities \( b_j \geq \epsilon k \) and \( c_{j+1} \geq \epsilon k \) hold.
	Then \( D_{\min}(\X) \geq \epsilon n / d \).
\end{proposition}

Given a permutation \( g \in \Sym(V) \) define \( \supp(g) = \{ x \in V \mid x^g \neq x \} \).
The minimal degree of a permutation group \( H \leq \Sym(V) \) is the minimum over nonidentity \( g \in H \) of \( |\supp(g)| \),
i.e.\ it is the minimal number of points moved by a nonidentity permutation of \( H \).

For a graph \( G \) (a coherent configuration \( \X \)) the \emph{motion} is the minimal degree of the full automorphism
group of \( G \) (or \( \X \), respectively). We denote it by \( \motion(G) \) or \( \motion(\X) \).
If \( G \) is distance-regular, then the full automorphism groups of \( G \) and the associated coherent configuration \( \X \)
coincide, so \( \motion(G) = \motion(\X) \) in this case.

\begin{proposition}[{\cite[Exercise~9.5]{babaiUpd}}]\label{motion}
	For a coherent configuration \( \X \) one has \( \motion(\X) \geq D_{\min}(\X) \).
\end{proposition}

Given a regular graph \( G \) of valency \( k \), let \( k = \xi_1 \geq \dots \geq \xi_m \) be its eigenvalues.
The \emph{zero-weight spectral radius} of \( G \) is \( \xi = \max_{i = 2, \dots, m} |\xi_i| \). For a distance-regular graph,
\( \xi = \max \{ \theta, \, -m \} \).

\begin{proposition}[{\cite[Proposition~12]{babaiAut}}]\label{zerow}
	Let \( G \) be a regular graph of valency \( k \) on \( n \) vertices
	with the zero-weight spectral radius \( \xi \). If every pair of vertices has at most \( q \)
	common neighbors in \( G \), then \( \motion(G) \geq n(k - \xi - q)/k \).
\end{proposition}

The next result is a slight modification of~\cite[Theorem~5.7]{kivvaJH}, but the proof stays essentially the same.
The only difference is that we use a stronger condition \( \xi \leq k(1 - \eta) \) instead of assuming that
\( c_j \leq \epsilon k \), \( b_j \leq \epsilon k \) for some \( j \in \{ 1, \dots, d-1 \} \).

\begin{proposition}[{\cite[Theorem~5.7]{kivvaJH}}]\label{dual}
	Let \( G \) be a Delsarte-geometric distance-regular graph of diameter \( d \geq 2 \)
	on \( n \) vertices. Suppose \( \mu = 1 \) and the smallest eigenvalue of \( G \) is \( m \), where \( -m \geq 3 \).
	Let \( 0 < \eta < 1/2 \) be some constant such that the zero-weight spectral radius satisfies \( \xi \leq k(1 - \eta) \)
	and \( k \geq \max \{ 4(-m)/\eta,\, m^2 \} \). Then \( \motion(G) \geq n\eta/4 \).
\end{proposition}

Next proposition covers the case when \( \mu = 1 \) and \( m > -3 \).

\begin{proposition}[{\cite[Proposition~5.13]{kivvaJH}}]\label{m2}
	Let \( G \) be a Delsarte-geometric distance regular graph of diameter at least~\( 2 \).
	Suppose that \( \mu = 1 \), \( k > 4 \) and the smallest eigenvalue of \( G \) is \( -2 \).
	Then \( \motion(G) \geq n/16 \).
\end{proposition}

If a distance-regular graph \( G \) with diameter \( d \geq 2 \) on \( n \) vertices is bipartite (and thus necessarily imprimitive),
then its distance-2 graph \( G_2 \) has two components, and the graphs induced on these components by \( G_2 \) are called its \emph{halved} graphs.
The halved graphs are distance-regular with diameter \( \lfloor d/2 \rfloor \) on \( n/2 \) vertices~\cite[Proposition~4.2.2]{brouwer}.

\begin{proposition}[{\cite[Proposition~5.7]{kivvaGap}}]\label{bip}
	Let \( G \) be a bipartite distance-regular graph of diameter \( d \geq 3 \). Let \( G^+ \) and \( G^- \) be
	the halved graphs of \( G \). Then \( \motion(G) \geq \motion(G^+) + \motion(G^-) \).
\end{proposition}

If the distance-\( d \) graph \( G_d \) of a distance-regular graph \( G \) of diameter \( d \geq 2 \) is a union of cliques,
then \( G \) is called an \emph{antipodal} graph. A \emph{folded} graph of \( G \) is a graph whose vertices are the cliques of \( G_d \)
and two vertices are adjacent in the folded graph if the corresponding cliques have vertices adjacent in~\( G \). The folded graph
is distance-regular of diameter \( \lfloor d/2 \rfloor \)~\cite[Proposition~4.2.2]{brouwer}.

\begin{proposition}[{\cite[Proposition~5.3]{kivvaGap}}]\label{anp}
	Let \( G \) be an antipodal distance-regular graph of diameter \( d \geq 3 \) on~\( n \) vertices.
	Let \( \widetilde{G} \) be its folded graph on \( \widetilde{n} \) vertices. If \( \motion(\widetilde{G}) \geq \alpha \widetilde{n} \)
	for some \( \alpha > 0 \), then \( \motion(G) \geq \alpha n \).
\end{proposition}

It is known that an imprimitive distance-regular graph of valency greater than two is bipartite or antipodal (or both)~\cite[Theorem~4.2.1]{brouwer},
but we will require a more detailed description.

\begin{proposition}[{\cite[Proposition~2.14]{kivvaGap}, see also~\cite[Section~4.2.A]{brouwer}}]\label{impclass}
	Let \( G \) be a distance-regular graph of valency greater than two. Then:
	\begin{enumerate}
		\item If \( G \) is a bipartite graph, then its halved graphs are not bipartite;
		\item If \( G \) is bipartite and either has odd diameter or is not antipodal, then its halved graph is primitive;
		\item If \( G \) is antipodal and either has odd diameter or is not bipartite, then its folded graph is primitive;
		\item If \( G \) has even diameter and is both antipodal and bipartite, then its halved graphs \( \frac{1}{2}G \)
			are antipodal, the folded graph \( \widetilde{G} \) is bipartite and the graphs \( \widetilde{\frac{1}{2}G} \simeq \frac{1}{2}\widetilde{G} \) are primitive.
	\end{enumerate}
\end{proposition}

\subsection{Characterizations of Johnson and Hamming graphs}\label{pJH}

For \( d \geq 2 \) and \( s \geq 2d \) the \emph{Johnson graph} \( J(s, d) \)
is defined as the graph having all \( d \)-element subsets of \( \{ 1, \dots, s \} \) as a vertex set,
where two vertices are adjacent if and only if the corresponding subsets differ in exactly one element.
This graph is distance-regular of diameter \( d \) and valency \( d(s-d) \).

Given a vertex \( v \) of the graph \( G \), let \( G(v) \) denote the neighborhood graph, i.e.\ the graph
induced by \( G \) on the neighbors of \( v \).

\begin{proposition}[{\cite[Theorem~1.2]{kivvaJH}}]\label{johnson}
	There exists an absolute constant \( \epsilon > 0.0065 \) such that the following is true.
	Let \( G \) be a Delsarte-geometric distance-regular graph of diameter \( d \geq 2 \)
	with the smallest eigenvalue \( m \). Suppose that \( \mu \geq 2 \) and \( \theta + 1 > (1 - \epsilon)b_1 \).
	If the valency satisfies \( k \geq \max \{ |m|^3,\, 29 \} \) and
	the neighborhood graph \( G(v) \) is connected for some vertex \( v \), then \( G \) is a Johnson graph \( J(s, d) \)
	with \( s = (k/d) + d \).
\end{proposition}

In a Delsarte-geometric distance-regular graph the neighborhood graphs are either all connected or all disconnected~\cite[Theorem~5.3]{bangGrs}.

\begin{proposition}[{\cite[Proposition~3.11]{kivvaJH}}]\label{ndis}
	Let \( G \) be a Delsarte-geometric distance-regular graph of diameter greater than two and with \( \mu \geq 3 \). 
	If the neighborhood graphs of \( G \) are disconnected, then \( \theta + 1 \leq 5b_1 / 7 \).
\end{proposition}

For \( d, s \geq 2 \) the \emph{Hamming graph} has \( \{ 1, \dots, s \}^d \) as the vertex set,
and two vertices are adjacent if and only if the corresponding \( d \)-tuples differ in one coordinate only.
This graph is distance-regular with diameter \( d \) and valency \( d(s-1) \).

\begin{proposition}[{\cite[Corollary~4.8]{kivvaJH}}]\label{hamming}
	Let \( G \) be a Delsarte-geometric distance-regular graph of diameter \( d \geq 2 \).
	Suppose that \( G \) has \( \mu = 2 \) and the smallest eigenvalue \( m \). Take \( 0 < \epsilon < 1/(6m^4d) \).
	If \( \theta \geq (1 - \epsilon)b_1 \) and \( c_i \leq \epsilon k \), \( b_i \leq \epsilon k \) for some \( i \in \{ 1, \dots, d \} \),
	then \( G \) is a Hamming graph \( H(d, s) \) for \( s = (k/d) + 1 \).
\end{proposition}

\section{Proof of Theorems~\ref{main} and~\ref{geom}}\label{pgraphs}

We will prove both theorems from the title simultaneously. Theorem~\ref{geom} will be an intermediate step in the proof of Theorem~\ref{main}.

Let \( G \) be a primitive distance-regular graph of valency \( k \) on \( n \) points, where \( n \) is large enough
(we may assume that by choosing the constant \( C \) appropriately).
Since cycle graphs have large motion, we may assume that \( k \geq 3 \). Let \( d \geq 3 \) denote the diameter of \( G \). 
We first consider the case when \( G \) is primitive.

Assume that \( c_2 > k/(20 d^4) \). Then
\[ 20 d^4 > \frac{k}{c_2} = \frac{b_0}{c_2} > \frac{b_1}{c_2} = \frac{k_2}{k_1} \geq \dots \geq \frac{k_{d}}{k_{d-1}}. \]
If \( \X \) is a coherent configuration associated with \( G \), then the valencies of its nondiagonal relations are \( k_1, \dots, k_d \).
If \( k_{\max} = k_i \), \( i \in \{ 2, \dots, d \} \), is the maximal valency, then \( k_{i} < k_{i-1} \cdot 20d^4 \)
and hence \( k_{\max}/(20 d^4) < n - k_{\max} \). By Proposition~\ref{dmin}, we have \( D_{\min} \geq k_{\max}/(20 d^5) \).

If \( k_{\max} \geq n/2 \), then \( D_{\min} \geq n/(40 d^5) \) and we have a bound on motion by Proposition~\ref{motion}.
If \( k_{\max} < n/2 \), then \( D_{\min} \geq (n - k_{\max})/d > n/(2d) \) and we again have a similar bound.
From now one we assume that \( c_2 \leq k/(20 d^4) \). In particular, \( k \geq 20 d^4 \).

If \( \motion(G) > n/2 \) then we are done, so let \( g \) be a nonidentity automorphism of \( g \) with \( |\supp(g)| \leq n/2 \).
For \( S = \supp(g) \) Proposition~\ref{expand} gives \( \frac{|\delta_G(S)|}{|S|} \geq \frac{k}{2d} \),
so there exists a vertex \( x \in S \) with at least \( k/(2d) \) neighbors in~\( V \setminus S \). Vertices \( x \) and \( x^g \)
are distinct and have at least \( k / (2d) \) common neighbors. If \( x \) and \( x^g \) are not adjacent,
then \( c_2 \geq k / (2d) \) which contradicts our assumption. Therefore they are adjacent and thus \( \lambda = a_1 \geq k / (2d) \).

The graph \( G \) is sub-amply regular with \( \mu = c_2 \leq k/(20 d^4) \). Then \( k \mu \leq k^2 / (20 d^4) \) while \( \lambda^2 \geq k^2 / (4d^2) \).
Therefore
\[ \lambda^2 \geq \frac{k^2}{4 d^2} \geq \frac{4k^2}{20d^4} \geq 4k\mu \]
and by Proposition~\ref{clique} the graph \( G \) contains a clique of size at least \( \lambda/2 \).
By the Delsarte bound, \( \lambda/2 \leq 1 + k / (-m) \), thus
\[ -m \leq \frac{k}{\lambda/2 - 1} \leq \frac{k}{\frac{k}{4d} - 1} < 5d. \]
Then \( m^2 \mu \leq 25d^2 \frac{k}{20d^4} < \frac{k}{2d} \leq \lambda \),
and by Proposition~\ref{bang} the graph \( G \) is Delsarte-geometric. Note that \( \mu < \lambda \).

At this point of the proof we have established Theorem~\ref{geom}. Now we proceed with the proof of Theorem~\ref{main}.
We roughly follow the argument in~\cite[Theorem~5.16]{kivvaJH}.

Recall that \( 1 \leq \mu \leq k/(20d^4) \), so \( 20d^4 \leq k \) and \( d \) is bounded in terms of~\( k \).
For a graph of diameter \( d \) and valency \( k \) we have \( n \leq 1 + k^d \), hence \( n \) is bounded in terms of~\( k \).
In particular, by choosing \( n \) large enough, we may assume that \( k \) is larger than any fixed absolute constant
(we note that this holds in general for any distance-regular graph by the positive solution of the Bannai-Ito conjecture in~\cite{bangFixVal},
but we need not use this deep result in our argument).

We use the following positive parameter throughout the proof:
\[ \epsilon = \frac{1}{6 \cdot (5d)^4 \cdot d}. \]
By Proposition~\ref{bandc}, if \( b_j \geq \epsilon k \) and \( c_{j+1} \geq \epsilon k \) for some \( j \in \{ 1, \dots, d-1 \} \), then
\[ \motion(G) \geq \frac{\epsilon n}{d} \geq C \cdot \frac{n}{d^6} \]
for some universal constant \( C > 0 \).
If for all \( j \) the inequalities for \( b_j \) and \( c_{j+1} \) do not hold, then since \( b_j \) are nonincreasing and
\( c_j \) are nondecreasing, we can find some \( j \in \{ 1, \dots, d \} \) with \( b_j \leq \epsilon k \) and \( c_j \leq \epsilon k \).

Now we consider several cases.
\medskip

\textbf{Case \( \theta < (1 - \epsilon)b_1 \).} The zero-weight spectral radius of \( G \)
is equal to \( \xi = \max \{ \theta,\, -m \} \), consequently, \( \xi \leq \max \{ (1 - \epsilon)b_1,\, -m \} \).
By Lemma~\ref{mulambda}, \( 2\lambda \leq k + \mu \) hence
\[ b_1 = k - 1 - \lambda \geq \frac{k - \mu - 2}{2}. \]
As \( \mu \leq k/(20d^4) \), we have \( 2\mu + 4 < k \), therefore
\[ b_1 \geq \frac{k - \mu - 2}{2} > \frac{k}{4}, \]
and \( (1 - \epsilon)b_1 > (1 - \epsilon)k/4 \).
Since \( -m < 5d < k/4(1 - \epsilon) \), as \( k \geq 20d^4 \), we obtain \( \xi \leq (1 - \epsilon)b_1 \).

Every pair of vertices of \( G \) has at most \( \max \{ \lambda,\, \mu \} \) common neighbors. Since this maximum is equal to \( \lambda \),
Proposition~\ref{zerow} implies
\[ \motion(G) \geq n \frac{k - \xi - \lambda}{k} \geq n \frac{k - (1 - \epsilon)b_1 - \lambda}{k}. \]
Now \( (1 - \epsilon)b_1 + \lambda \leq k - \epsilon b_1 \leq (1 - \epsilon/4)k \), thus \( \motion(G) \geq n \epsilon/4 \geq C \cdot n/d^5 \)
for some universal constant \( C > 0 \).
\medskip

\textbf{Case \( \theta \geq (1 - \epsilon)b_1 \) and \( \mu \geq 3 \).}
We claim that for \( n \) sufficiently large we have \( \max \{ |m|^3,\, 29 \} \leq k \).
Since \( |m|^3 < (5d)^3 \) and \( k \geq 20d^4 \), this holds for \( d \geq 7 \). If \( d < 7 \),
then this holds as \( k > (5 \cdot 6)^3 \) for \( n \) large enough.

Suppose that the neighborhood graphs of \( G \) are disconnected, then by Proposition~\ref{ndis},
we have \( \theta + 1 \leq 5b_1/7 \). Then \( 5b_1/7 > (1 - \epsilon)b_1 \) implying \( \epsilon > 2/7 \), which is a contradiction.
We may assume that \( \epsilon < 0.0065 \), therefore the neighborhood graphs of \( G \) are connected and Proposition~\ref{johnson} shows that \( G \) is a Johnson graph.
\medskip

\textbf{Case \( \theta \geq (1 - \epsilon)b_1 \) and \( \mu = 2 \).}
Since \( b_j \leq \epsilon k \) and \( c_j \leq \epsilon k \) for some \( j \in \{ 1, \dots, d-1 \} \),
and \( \epsilon < 1/(6m^4d) \), Proposition~\ref{hamming} implies that \( G \) is a Hamming graph.
\medskip

\textbf{Case \( \mu = 1 \).} Assume that \( -m \geq 3 \). By Proposition~\ref{eigenexp}, \( \theta \leq k (1 - 1/(8d^2)) \).
As \( -m \leq 5d \) and \( k \geq 20d^4 \), we have a bound on the zero-weight spectral radius:
\[ \xi = \max \{ \theta,\, -m \} \leq k \left(1 - \frac{1}{8d^2} \right). \]
Set \( \eta = 1/(8d^2) \), so \( \xi \leq k(1 - \eta) \). Clearly \( k > m^2 \) and \( k > 4m/\eta = 32(-m)d^2 \).
Proposition~\ref{dual} implies \( \motion(G) \geq n\eta/4 = n/(32 d^2) \) as wanted. 

If \( -m < 3 \), then by Lemma~\ref{mint}, \( m \) is an integer and every vertex lies in precisely \( -m \) cliques from our geometry, hence \( m = -2 \).
Then by Proposition~\ref{m2} we have \( \motion(G) \geq n/16 \). This finishes the proof of the primitive case of Theorem~\ref{main},
namely, we established the following.

\begin{proposition}\label{mainprim}
	Let \( G \) be a primitive distance-regular graph on \( n \) points of diameter~\( d \geq 3 \).
	Then either \( G \) is a Johnson or a Hamming graph, or
	\[ \motion(G) \geq \gamma_d n \]
	where \( \gamma_d = C/d^6 \) for some universal constant \( C > 0 \).
\end{proposition}

We now assume that \( G \) is an imprimitive distance-regular graph. We need the following lemma, which is an improvement of~\cite[Theorem~5.8]{kivvaGap}
but with a better constant.
\begin{lemma}\label{lembip}
	Let \( G \) be a distance-regular bipartite graph of diameter \( d \geq 4 \) on \( n \) vertices.
	If a halved graph of \( G \) is primitive, then
	\[ \motion(G) \geq \gamma'_d n \]
	for \( \gamma'_d = \gamma_{\lfloor d/2 \rfloor}/2 \), where \( \gamma_d \) is a function from Proposition~\ref{mainprim}.
\end{lemma}
\begin{proof}
	Let \( G' \) be a primitive halved graph of \( G \). By Proposition~\ref{bip}, \( \motion(G) \geq \motion(G') \),
	and since \( G' \) has diameter \( \lfloor d/2 \rfloor \) and \( n/2 \) vertices, the claim follows from Proposition~\ref{mainprim}.
\end{proof}

The rest of the argument closely follows~\cite[Theorem~5.16]{kivvaGap}, but uses better constants. We mention that ``Conjecture~1.5'' in~\cite{kivvaGap}
corresponds to our Proposition~\ref{mainprim}.

Recall that an imprimitive distance-regular graph is bipartite or antipodal, or both. If \( G \) is bipartite but not antipodal and \( d \geq 4 \),
then by Proposition~\ref{impclass}~(2) its halved graph is primitive and hence \( \motion(G) \geq \gamma'_d n \) by Lemma~\ref{lembip}.
If \( G \) is bipartite of diameter \( d = 3 \), then by~\cite[Theorem~5.12]{kivvaGap} the graph \( G \) is either a crown graph or \( \motion(G) \geq n/6 \).

If \( G \) is bipartite and antipodal of even diameter \( d \geq 6 \), then by Proposition~\ref{impclass}~(4) its folded graph \( \widetilde{G} \)
is bipartite (and not antipodal) of diameter~\( d/2 \). A crown graph is antipodal, hence the previous paragraph implies
\[ \motion(\widetilde{G}) \geq \min(\gamma'_d, 1/6) \widetilde{n}, \]
where \( \widetilde{n} \) is the number of vertices in~\( \widetilde{G} \). Then by Proposition~\ref{anp}
\[ \motion(G) \geq \min(\gamma'_d, 1/6) n. \]
If \( G \) is bipartite and antipodal of diameter \( d = 4 \), then by~\cite[Proposition~5.10]{kivvaGap}, \( \motion(G) \geq 0.15 n \).

We are left with cases when \( G \) is antipodal but not bipartite, and when \( G \) is antipodal of odd diameter. By Proposition~\ref{impclass}~(3),
the folded graph \( \widetilde{G} \) is primitive. If \( d = 3 \), then \( \motion(G) \geq n/13 \) by~\cite[Proposition~5.14]{kivvaGap}.
If \( d \geq 4 \) and \( \widetilde{G} \) is not a Johnson graph, a Hamming graph or a complement to \( J(s, 2) \) or \( H(2, s) \),
and \( \widetilde{G} \) has at least 29 vertices, then Proposition~\ref{mainprim} and Proposition~\ref{babaiSRG} imply
\[ \motion(\widetilde{G}) \geq \min(\gamma_{\lfloor d/2 \rfloor}, 1/8) \widetilde{n}, \]
where \( \widetilde{n} \) is the number of vertices in~\( \widetilde{G} \). Then by Proposition~\ref{anp}
\[ \motion(G) \geq \min(\gamma_{\lfloor d/2 \rfloor}, 1/8) n. \]

Finally, it is easy to see that if \( \widetilde{G} \) has at most \( 28 \) vertices, then \( \motion(G) \geq n/14 \),
and it follows from~\cite[Theorem~5.15]{kivvaGap} that a Johnson graph, a Hamming graph, and complements to \( J(s, 2) \) or \( H(2, s) \) have no antipodal covers
(taking into account our bound on the number of vertices).
Hence we proved that \( \motion(G) \geq \gamma n \) for
\[ \gamma = \min \left(\gamma_d, \gamma'_d, \gamma'_{\lfloor d/2 \rfloor}, \gamma_{\lfloor d/2 \rfloor}, \frac{1}{14} \right). \]
For an appropriate universal constant \( C > 0 \) we have \( \gamma \geq C/d^6 \), so the proof of Theorem~\ref{main} is finished.

\section{Proof of Theorem~\ref{maingroup}}\label{pgroups}

For a permutation group \( H \) on \( V \) and a subset \( \Sigma \subseteq V \) let \( H_\Sigma \) denote the pointwise stabilizer of \( \Sigma \).
Given a graph \( G \) with vertex set \( V \), we say that \( \Sigma \subseteq V \) \emph{splits} an edge \( e \) with respect to \( H \),
if the two endvertices of \( e \) lie in different \( H_\Sigma \)-orbits on~\( V \).

\begin{proposition}[{\cite[Splitting Lemma~3.1]{babai2Tr}}]\label{splitlem}
	Let \( G \) be a graph on \( V \) having \( q \) edges.
	Assume that a permutation group \( H \) acts edge-transitively on \( G \).
	If \( \Sigma \subseteq V \) splits at least \( cq \), \( 0 < c < 1 \), edges with respect to \( H \), then there exists
	\( \Delta \subseteq V \) which splits every edge of \( G \) with respect to \( H \) and \( |\Delta| \leq |\Sigma| \cdot \lceil \frac{-\log q}{\log (1-c)} \rceil \).
\end{proposition}

Given a permutation group \( H \) on \( V \), we say that \( \Sigma \subseteq V \) is a \emph{halving set} if \( H_\Sigma \) has no orbits of size greater than \( |V|/2 \).
Recall that the \emph{thickness} \( \theta(H) \) of a group \( H \) is the largest degree of an alternating section of \( H \).

\begin{proposition}[{\cite[Lemma~3]{pyber2tr}}]\label{halvinglem}
	Let \( H \) be a permutation group on \( V \), and suppose that \( \theta(H) < k \) for some \( k \geq 12 \).
	Then there is a halving set \( \Sigma \subseteq V \) for \( H \) with \( |\Sigma| \leq 4k \).
\end{proposition}

We will require the following corollary of Wielandt's bound for thickness~\cite{wielandtD}.
\begin{proposition}[{\cite[Corollary~6.2]{kivvaJH}}]\label{thetaIneq}
	Let \( H \) be a permutation group of degree \( n \) with minimal degree at least \( \alpha n \).
	Then \( \theta(H) \leq \frac{3}{\alpha} \log n \).
\end{proposition}

The following lemma is key to our proof. Recall that \( \{ x, y \} \) is an edge of a distance~2 graph \( G_2 \)
if and only if \( d(x, y) = 2 \).

\begin{lemma}\label{pybTrick}
	Let \( G \) be a connected graph on \( V \) and let \( H \) be its group of automorphisms.
	If a nonempty \( \Sigma \subseteq V \) splits all edges of \( G \) and \( G_2 \), then \( H_\Sigma \) is trivial.
\end{lemma}
\begin{proof}
	Let \( \Delta \subseteq V \) be the set of points stabilized by \( H_\Sigma \); it suffices show that \( \Delta = V \).

	Clearly \( \Sigma \subseteq \Delta \), hence \( \Delta \) is nonempty. Suppose that \( \Delta \) is a proper subset of \( V \),
	then by connectivity of \( G \) there exists some vertex \( x \in V \setminus \Delta \) connected by an edge with some \( y \in \Delta \).
	By the definition of \( \Delta \), the vertex \( x \) is not stabilized by \( H \), thus for some \( h \in H \) the vertex \( x' = x^h \) lies in \( V \setminus \Delta \)
	and is adjacent to \( y \). There is a path \( x, y, x' \) in \( G \), so \( d(x, x') \leq 2 \). By our assumptions, \( x \) and \( x' \) must lie
	in different \( H \)-orbits, which is a contradiction with the definition of \( x' \). Therefore \( \Delta = V \) and we are done.
\end{proof}

\noindent\emph{Proof of Theorem~\ref{maingroup}.}
Let \( G \) be a distance-regular graph on \( V \), where \( |V| = n \), of diameter greater than two.
Let \( H \) be its full automorphism group. By Theorem~\ref{main} and Proposition~\ref{diam}, the minimal degree of \( H \)
is at least \( Cn/(\log n)^6 \) for some universal constant \( C > 0 \). By Wielandt's bound,
\( \theta(H) \leq \frac{3}{C}(\log n)^7 \) and by Proposition~\ref{halvinglem}, there exists a halving set \( \Sigma \subseteq V \)
with \( |\Sigma| \leq C' (\log n)^7 \) for some universal constant \( C' > 0 \).

Let \( S \subseteq V \) be some orbit of \( H_\Sigma \) on \( V \). Clearly \( |S| \leq n/2 \)
and by Proposition~\ref{expand}, we have \( |\delta_G(S)|/|S| \geq k/(2d) \), where \( k \) is the valency of \( G \)
and \( d \) is the diameter. This implies that at least \( k|S| / (2d) \) edges of \( G \) with an endpoint in \( S \) are split by~\( \Sigma \).
Summing by all orbits of \( H \) and recalling that we may count the same edge twice, we get that at least \( kn/(4d) \) edges of \( G \)
are split by \( \Sigma \). If \( q = kn/2 \) is the number of edges of \( G \), then \( \Sigma \) splits at least \( q/(2d) \) edges of \( G \).

Since \( G \) is primitive, the distance~2 graph \( G_2 \) is also connected and its diameter is bounded by \( d \),
hence the same reasoning applies and \( \Sigma \) splits at least \( q_2/(2d) \) edges of \( G_2 \), where \( q_2 \) is the number of edges in \( G_2 \).
By Proposition~\ref{splitlem}, there are sets \( \Delta \) and \( \Delta_2 \) of size at most \( 4|\Sigma|d \log n \) which split every edge
of the graphs \( G \) and \( G_2 \), respectively. By taking \( \Sigma' = \Delta \cup \Delta_2 \) and recalling that \( d \leq 5 \log n \) by Proposition~\ref{diam},
we obtain that \( |\Sigma'| \leq C''(\log n)^9 \) for some universal constant \( C'' > 0 \), and \( \Sigma' \) splits all edges of \( G \) and \( G_2 \).
Lemma~\ref{pybTrick} shows that \( \Sigma' \) is a base for \( H \), which finishes the proof. \qed

\section{Acknowledgement}

The authors would like to thank B.~Kivva for the remarks which improved the exposition of this paper.


\end{document}